\documentclass{daj}
\usepackage{amsthm,amsmath,amssymb,amsfonts,nicefrac}

\dajAUTHORdetails{%
  title = {Isoperimetric Inequalities Made Simpler}, 
  author = {Ronen Eldan, Guy Kindler, Noam Lifshitz, and Dor Minzer},
  plaintextauthor = {Ronen Eldan, Guy Kindler, Noam Lifshitz, and Dor Minzer},
  keywords = {analysis of boolean functions, isoperimetric inequalities.},
}   

\dajEDITORdetails{%
   year={2025},
   number={7},
   received={14 February 2023},   
   revised={14 March 2024},    
   published={24 July 2025},  
   doi={10.19086/da.142095},       
}   

\newtheorem{thm}{Theorem}[section]

\newtheorem{lemma}[thm]{Lemma}
\newtheorem{corollary}[thm]{Corollary}
\newtheorem{claim}[thm]{Claim}

\newtheorem{definition}[thm]{Definition}
\newtheorem{remark}[thm]{Remark}

\newtheorem{fact}[thm]{Fact}

\newcommand\E{\mathop{\mathbb{E}}}
\newcommand\card[1]{\left| {#1} \right|}
\newcommand\sett[2]{\left\{ \left. #1 \;\right\vert #2 \right\}}

\newcommand\set[1]{{\left\{ #1 \right\}}}
\newcommand\Prob[2]{{\Pr_{#1}\left[ {#2} \right]}}

\newcommand\Expect[2]{{\mathop{\mathbb{E}}_{#1}\left[ {#2} \right]}}

\newcommand\norm[1]{\| #1 \|}
\newcommand\power[1]{\set{0,1}^{#1}}

\newcommand\defeq{\stackrel{def}{=}}

\newcommand\inner[2]{\langle{#1},{#2}\rangle}
\newcommand\eps{\varepsilon}

\def\gg{\gtrsim}
\def\ll{\lesssim}

\begin{document}

\begin{frontmatter}[classification=text]

\title{Isoperimetric Inequalities Made Simpler} 

\author[pgom]{Ronen Eldan}
\author[johan]{Guy Kindler}
\author[laci]{Noam Lifshitz\thanks{supported by ISF individual research grant 1980/22.}}
\author[andy]{Dor Minzer\thanks{Supported by a Sloan Research Fellowship and NSF CCF award 2227876.}}

\begin{abstract}
We give an alternative, simple method to prove isoperimetric inequalities over the hypercube. In particular, we show:
  \begin{enumerate}
    \item An elementary proof of classical isoperimetric inequalities of Talagrand, as well as a stronger isoperimetric result conjectured by Talagrand
    and recently proved by Eldan and Gross.
    \item A strengthening of the Friedgut junta theorem, asserting that if the $p$-moment of the sensitivity of a function is constant for some $1/2 + \eps\leq p\leq 1$,
    then the function is close to a junta. In this language, Friedgut's theorem is the special case that $p=1$.
  \end{enumerate}
\end{abstract}
\end{frontmatter}

\section{Introduction}
\subsection{Isoperimetric Inequalities Over the Hypercube}
Isoperimetric inequalities are fundamental results in mathematics with a myriad of applications. The main topic of this paper is discrete isoperimetric
inequalities, and more specifically isoperimetric inequalities over the hypercube $\{0,1\}^n$. Classical results along
these lines are due to Harper, Bernstein, Lindsey and Hart~\cite{Harper2,Bernstein,Lindsey,Hart}, Harper~\cite{Harper2}, Margulis~\cite{Margulis}, Talagrand~\cite{Tal1,Tal2,Tal4} and Kahn, Kalai and Linial~\cite{KKL}.

All of these theorems discuss various measures of boundaries and relations between them for Boolean functions.
For a function $f\colon\power{n}\to\power{}$ and a point $x\in\power{n}$, we define the \emph{sensitivity} of $f$ at $x$, denoted by $s_f(x)$,
to be the number of coordinates $i\in[n]$ such that $f(x\oplus e_i)\neq f(x)$. Thus, the \emph{vertex boundary} $V_f$ of a function $f$ is defined to be the set
of all $x$'s such that $s_f(x) > 0$, and the \emph{edge boundary} $E_f$ of $f$ is defined to be the set of edges $(x,x\oplus e_i)$ of the hypercube whose
endpoints are assigned different values by $f$. The most basic isoperimetric result related to Boolean functions,
known as Poincar\'{e}'s inequality, asserts that $\frac{\card{E_f}}{2^n}\geq {\sf var}(f)$,
i.e.\ the edge boundary of a function $f$ cannot be too small.\footnote{We remark that the term ``Poincar\'{e}'s inequality'' comes from the theory of functional inequalities
and Markov chains and refers to a much more general result, see for example~\cite[Section 2]{kitsos2014inequalities}. In the case of graphs, it asserts that for a graph
$G$, if the first non-trivial eigenvalue of its Laplacian $L$ is at least $\lambda$, then $\norm{L f}_2\geq \lambda \norm{f-\E[f]}_2$. In our
example, the Laplacian $L$ is $L = I-\frac{1}{n} A$ where $A$ is the adjacency matrix of the Boolean hypercube, and $\lambda$ is $2/n$.}
The edge isoperimetric inequality for the hypercube, due to Harper, Bernstein, Lindsey and Hart~\cite{Harper2,Bernstein,Lindsey,Hart}, is a strengthening
of that statement, asserting that for a fixed value of $\E[f]$, the functions that minimize the edge boundary $\frac{\card{E_f}}{2^n}$
are indicator functions
of initial segments in lexicographical orderings of the hypercube. In particular, this result implies that if $f$ is a non-constant Boolean function, then $\frac{\card{E_f}}{2^n}\geq \E[f]\log\left(\frac{1}{\E[f]}\right)$
(here and throughout, $\log$ means the logarithm function with base $2$). The vertex isoperimetric inequality, due to Harper~\cite{Harper2}, asserts that among all Boolean function
of a given measure, the functions that minimize the size of the vertex boundary $V_f$ are the indicator functions of
initial segments in simplicial orderings of the hypercube.
Margulis~\cite{Margulis} proved a result that combines the notion of edge boundary and vertex boundary, showing that for all functions $f$ one has that
$\frac{\card{V_f}}{2^n}\cdot \frac{\card{E_f}}{2^n}\geq \Omega({\sf var}(f)^2)$.

Talagrand considered a different notion of boundary for a function $f$, namely the quantity $\Expect{x}{\sqrt{s_f(x)}}$, which we will refer to as the \emph{Talagrand boundary} of $f$.
Here and throughout, the distribution over $x$ will be uniform over the hypercube $\power{n}$.
Talagrand proved a number of results about this quantity~\cite{Tal1,Tal2}, the most basic of which asserts that $\Expect{x}{\sqrt{s_f(x)}}\geq \Omega({\sf var}(f))$. Using a simple application of the Cauchy-Schwarz inequality one sees that the result of Talagrand implies Margulis' theorem. Talagrand then went on and strengthened the result for
functions with small variance, proving the following theorem:
\begin{thm}[Talagrand~\cite{Tal1}]\label{thm:Talagrand_ineq}
  For all non-constant Boolean functions $f\colon\{0,1\}^n\to\{0,1\}$ it holds that
  \[
    \Expect{x}{\sqrt{s_f(x)}}\geq\Omega\left({\sf var}(f)\sqrt{\log\left(\frac{1}{{\sf var}(f)}\right)}\right).
  \]
\end{thm}
We remark that Talagrand's result is in fact more general and slightly stronger. First, Talagrand proves that the above inequality holds for the $p$-biased measure of the
cube if one inserts a factor of $\frac{\min(p,1-p)}{\sqrt{p(1-p)}}$ to the right hand side. The result and techniques of the current paper also generalize to the $p$-biased measure in a straight-forward manner,
and we focus on the case that $p=1/2$ for simplicity. Second, Talagrand considers a variant of the quantity $s_f(x)$ wherein $s_f(x)$ is
defined to be $0$ for $x$ such that $f(x) = 0$. The results we state below also hold with this stronger definition.

While Theorem~\ref{thm:Talagrand_ineq} can be seen to be tight up to constant factors (as can be seen by considering sub-cubes), it leaves something to be desired: it is incomparable to yet another prominent isoperimetric result -- the KKL theorem~\cite{KKL}.
The KKL Theorem is concerned with the influence of coordinates on a function. For a coordinate $i\in [n]$, the \emph{influence} of $i$ is defined by
\[
I_i[f] = \Prob{x\in\power{n}}{f(x)\neq f(x\oplus e_i)}.
\]
In words, it is the fraction of edges along direction $i$ that are in the edge boundary of $f$.
The KKL theorem~\cite{KKL} asserts that any function must have a variable that is at least somewhat influential; more quantitatively, they proved that
$\max_{i} I_i[f]\geq \Omega\left(\frac{\log n}{n} {\sf var}(f)\right)$.

Since the isoperimetric results of Talagrand and of KKL look very different (as well as the techniques that go into their proofs), Talagrand asked whether
one can establish a single isoperimetric inequality that is strong enough to capture both of these results simultaneously. In~\cite{Tal2}, Talagrand conjectures
such a statement, and makes some progress towards it. Namely, he shows that there exists $\alpha\in(0,1/2]$ such that for
any non-constant function $f\colon\power{n}\to\power{}$ it holds that
\begin{equation}\label{eq2}
\Expect{x}{\sqrt{s_f(x)}}\geq c\cdot {\sf var}(f)\left(\log\left(\frac{1}{{\sf var}(f)}\right)\right)^{1/2-\alpha}
\left(\log\left(1+\frac{1}{\sum\limits_{i=1}^{n} I_i[f]^2}\right)\right)^{\alpha},
\end{equation}
for some absolute constant $c>0$.
This inequality can be seen to imply a weaker result along the lines of the KKL theorem,\footnote{Namely that there is $\beta>0$ such that if ${\sf var}(f)\geq \Omega(1)$,
$\max_{i} I_i[f]\geq\Omega((\log n)^{\beta}/n)$.} but is not enough to fully recover it. Talagrand conjectured that the above inequality holds for
$\alpha = 1/2$, a statement that is strong enough to directly imply the KKL theorem.

Recently, Eldan and Gross~\cite{EG} introduced techniques from stochastic processes to the field, and used these to prove that Talagrand's conjecture is true, i.e.~that one may take $\alpha = 1/2$ in~\eqref{eq2}.

\subsection{Our Results}
We show a unified and simple Fourier analytic approach, based on random restrictions, to establish the results mentioned above, as well as a strengthening of Friedgut's junta theorem~\cite{Friedgut}. Our results are:
\begin{enumerate}
  \item {\bf Classical isoperimetric inequalities.} We give Fourier analytic proofs of the inequalities above by Margulis, Talagrand and Eldan-Gross,
  as well as of robust versions of them.
  Our proof technique also unravels a connection between these isoperimetric inequalities and other classical results in the area, such as the
  Friedgut junta theorem~\cite{Friedgut} and Bourgain's tail theorem~\cite{Bourgain}. The fact that such purely Fourier analytic proof exists
  is quite surprising to us, as the quantity $\Expect{x}{\sqrt{s_f(x)}}$ does not seem to relate well to Fourier analytic methods (and indeed,
  the proof of Talagrand proceeds by induction, whereas the proof of Eldan and Gross uses stochastic calculus).

  \item {\bf Strengthened junta theorems.} We establish a strengthening of Friedgut's theorem in which the assumption that the average sensitivity of $f$ is small,
  i.e. that $\Expect{x}{s_f(x)}$ is small, is replaced by the weaker condition that $\Expect{x}{s_f(x)^p}$ is small, for $p$ slightly larger than $1/2$.
\end{enumerate}

In the rest of this introductory section, we describe our results in more detail.

\subsection{Classical Isoperimetric Inequalities and Robust Versions}
We give simpler proofs for the results below, which are due to Talagrand~\cite{Tal1} and Eldan-Gross~\cite{EG}.
\begin{thm}\label{thm:tal_iso}
  There exists $c>0$ such that for all non-constant functions $f\colon\power{n}\to\power{}$ we have that
  \[
  \Expect{x}{\sqrt{s_f(x)}}\geq c\cdot {\sf var}(f)\sqrt{\log\left(\frac{1}{{\sf var}(f)}\right)}.
  \]
\end{thm}

\begin{thm}\label{thm:EG}
  There exists $c>0$ such that for all non-constant functions $f\colon\power{n}\to\power{}$ we have that
  \[
  \Expect{x}{\sqrt{s_f(x)}}\geq c\cdot {\sf var}(f)\sqrt{\log\left(1+\frac{1}{\sum\limits_{i=1}^{n} I_i[f]^2}\right)}.
  \]
\end{thm}
Our technique is more general, and can be used to prove more robust versions of Theorems~\ref{thm:tal_iso},~\ref{thm:EG}.
For a $2$-coloring of the edges of the hypercube ${\sf col}\colon E(\{0,1\}^n)\to \{{\sf red}, {\sf blue}\}$,
define the red sensitivity as $s_{f,{\sf red}}(x) = 0$ if $f(x) = 0$, and otherwise $s_{f,{\sf red}}(x)$ is the number of sensitive edges
adjacent to $x$ that are red. Similarly, define the blue sensitivity as $s_{f,{\sf blue}}(x) = 0$ if $f(x) = 1$ and otherwise
$s_{f,{\sf blue}}(x)$ is the number of sensitive edges adjacent to $x$ that are blue. Then, defining
\[
{\sf Talagrand}_{{\sf col}}(f)
=
\Expect{x}{\sqrt{s_{f,{\sf red}}(x)}}
+\Expect{x}{\sqrt{s_{f,{\sf blue}}(x)}}
\]
we have that the left hand side of Theorems~\ref{thm:tal_iso},~\ref{thm:EG} can be replaced by
${\sf Talagrand}_{{\sf col}}(f)$ for any $2$-coloring ${\sf col}$. Such results can be used
to prove the existence of nearly bi-regular structures in the sensitivity graph of $f$ which
are sometimes useful in applications (for example, in~\cite{KMSmono} such structures in
the directed sensitivity graph are necessary to construct monotonicity testers).
To get some intuition, consider the bipartite sub-graph of the hypercube consisting
of sensitive edges of $f$, and suppose that we want to find a large sub-graph of it
which is nearly regular, in the sense that the degrees of one side are between $d$ and $2d$, while
the degrees of the other side are at most $d$ for some $d$. In that case, taking the coloring ${\sf col}_{{\sf red}}$ that assigns
to all edges red or the coloring ${\sf col}_{{\sf blue}}$ that assigns all of the edges blue (which amount to Talagrand's original formulation)
gives us information about the typical degrees of vertices $x$ with $f(x) = 1$ and vertices $y$ with $f(y) = 0$ respectively.
It doesn't give us any information about the possible relation between the degrees of such $x$'s and such $y$'s, though.
Using the coloring ${\sf col}$ appropriately, one may establish a better balance between these two degrees and get useful
information on the relation between them. Establishing such robust results using the techniques of~\cite{EG} is more challenging.

\subsection{New Junta Theorems}
Next, we study connections between Talagrand-like quantities such as $\Expect{x}{\sqrt{s_f(x)}}$, and the notion of juntas.
For a point $x\in \{0,1\}^n$ and a subset $J\subseteq [n]$, we denote by $x_J$ the point in $\{0,1\}^J$ corresponding to
the value of $x$ on coordinates $J$.
\begin{definition}
  A function $f\colon\{0,1\}^n\to\{0,1\}$ is called a $j$-junta if there exists $J\subseteq [n]$ of size at most $j$,
  and $g\colon \{0,1\}^J\to\{0,1\}$ such that $f(x) = g(x_J)$ for all $x\in\{0,1\}^n$.

  We say a function $f\colon\{0,1\}^n\to\{0,1\}$ is $\eps$-close to a $j$-junta if there is a $j$-junta $g\colon\{0,1\}^n\to\{0,1\}$
  such that $\Prob{x}{f(x)\neq g(x)}\leq \eps$.
\end{definition}
On the one hand, it is possible that the quantity $\Expect{x}{\sqrt{s_f(x)}}$ is constant while the function $f$ is very far from being a junta,
as can be seen by taking $f$ to be the majority function. On the other hand, if we look at the expected sensitivity -- that is
$\Expect{x}{s_f(x)}$ -- a well known result of Friedgut~\cite{Friedgut} asserts that if $\Expect{x}{s_f(x)} = O(1)$, then $f$
is close to junta. More precisely, Friedgut's result asserts that if $\Expect{x}{s_f(x)} \leq A$, then
for all $\eps>0$, $f$ is $\eps$-close to a $2^{O(A/\eps)}$-junta. In light of this contrast, it is interesting to ask
what happens if an intermediate moment of the sensitivity, say the $p$th moment for $p\in (1/2,1)$, is constant.

We prove a strengthening of Friedgut's junta theorem~\cite{Friedgut} addressing this case, morally showing that for any $p>1/2$, bounded $p$-moment of
the sensitivity implies closeness to junta. More precisely:
\begin{thm}\label{thm:super_fried}
  For all $\eta>0$ there is $C = C(\eta)>0$ such that  the following holds.
  For all $\eps>0$ and $A>0$, and $1/2 + \eta\leq p\leq 1$, if $f\colon\{0,1\}^n\to\{0,1\}$ is a function such that
  $\Expect{x}{s_f(x)^{p}}\leq A$, then $f$ is $\eps$-close to a $J$-junta, where $J = 2^{C\left(A/\eps\right)^{1/p}}$.
\end{thm}
Theorem~\ref{thm:super_fried} is tight, and can be seen as a generalization of Friedgut's Junta Theorem (corresponding to the case that
$p=1$) that is somewhat close in spirit to Bourgain's Junta Theorem. Indeed, for all $p\geq 1/2$ the tribes function shows
this to be tight. Here, the tribes function is the function in which we take a partition $P_1\cup\ldots\cup P_m$ of $[n]$
and define ${\sf Tribes}(x) = 1$ if there is $i$ such that $x_{P_{i}} = \vec{1}$, and ${\sf Tribes}(x) = 0$ otherwise.
We take $P_i$ of equal size and the probability that ${\sf Tribes}(x)=1$ is roughly $1/2$; the correct choice
of parameters is $\card{P_i} = \log n - \log\log n + c$
and $m =  \frac{n}{\log n - \log\log n + c}$ where $c = O(1)$.
For this function one has $\Expect{x}{s_{{\sf Tribes}}(x)^{p}} = \Theta((\log n)^p)$. Indeed, as $s_{{\sf Tribes}}(x) = \Theta(\log n)$ with
constant probability we have the lower bound, and the upper bound follows by Jensen's inequality and the fact that the total influence of ${\sf Tribes}$
is $\Theta(\log n)$.

As a consequence, Theorem~\ref{thm:super_fried} also implies a version of the KKL theorem for $p$th moments of sensitivity:
\begin{corollary}\label{thm:KKL_varaint}
  For all $\delta>0$, there is $c>0$ such that the following holds for all $p\in [1/2+\delta,1]$.
  Suppose that a function $f\colon \{0,1\}^n\to \{0,1\}$ satisfies that $\Expect{x}{s_f(x)^{p}}\leq A$, then there exists $i\in [n]$ such that
  \[
  I_i[f]\geq 2^{-c \left(A/{\sf var}(f)\right)^{1/p}}.
  \]
\end{corollary}
\begin{proof}
  Assume without loss of generality that $\E[f]\leq \frac{1}{2}$ (otherwise we work with $1-f$).
  Take $\eps = \E[f]/10$ and apply Theorem~\ref{thm:super_fried} to get that $f$ is $\eps$-close to a function
  $g$ which is a $J$-junta for $J = 2^{\Theta(\left(A/{\sf var}(f)\right)^{1/p})}$, and let $\mathcal{J}$ be the set
  of coordinates that $g$ depends on. Choose $x\in \{0,1\}^n$ uniformly, and take $y\in\{0,1\}^n$ where $y_{\overline{\mathcal{J}}} = x_{\overline{\mathcal{J}}}$
  and $y_{\mathcal{J}}$ is sampled uniformly from $\power{\mathcal{J}}$. Then $\Prob{x,y}{g(x)=0, g(y)=1} = \E[g](1-\E[g])$; also, $f(x) = g(x)$ and $f(y) = g(y)$ hold with
  probability at least $1-2\eps$, hence we conclude that
  \[
  \Prob{x,y}{f(x)=0, f(y)=1} \geq \E[g](1-\E[g]) - 2\eps \geq \frac{1}{2}(\E[f]-\eps) - 2\eps \geq \frac{1}{4}\E[f].
  \]
  On the other hand, write $\mathcal{J} = \set{i_1,\ldots,i_J}$ and consider a path $x(0),\ldots,x(J)$ where $x(0) = x$, $x(J) = y$ and $x(t), x(t+1)$
  may differ only on coordinate $i_{t+1}$. Then by the union bound,
  \[
  \frac{1}{4}\E[f]\leq
  \Prob{x,y}{f(x)=0, f(y)=1}
  \leq \sum\limits_{t=0}^{J-1}\Prob{x,y}{f(x(t))=0, f(x(t+1))=1}
  =\sum\limits_{t=0}^{J-1} I_{i_{t+1}}[f]
  \leq J\max_i I_i[f],
  \]
  and the result follows.
\end{proof}
We note that the KKL theorem is this statement for $p=1$, and the version for $p<1$ is only stronger than the KKL theorem as always
$\Expect{x}{s_f(x)^{p}}
\leq \Expect{x}{s_f(x)}^p$.

\begin{remark}
We remark that a naive application of our method establishes a sub-optimal dependency between the parameters $A$ and $J$ in Theorem~\ref{thm:super_fried},
but gives an interesting relation to the noise sensitivity theorem of Bourgain~\cite{Bourgain}. More precisely, in Corollary~\ref{corr:noise_stable} we show that if
$f\colon\power{n}\to\power{}$ satisfies $\Expect{x}{s_f(x)^p}\leq A$, then $f$ must be noise stable.
More specifically, for $\eps>0$, we define the $(1-\eps)$-correlated distribution $(x,y)$
by taking $x\in \power{n}$ uniformly, and for each $i\in[n]$ independently, setting $y_i = x_i$ with probability $1-\eps$
and otherwise resampling $y_i\in\power{}$ uniformly. In this language, we show that if $\Expect{x}{s_f(x)^p}\leq A$, and $p<1$,
then for all $\eps>0$,
\[
\Prob{x,y\text{ $(1-\eps)$-correlated}}{f(x) = f(y)} \geq 1-O_p(A\cdot \eps^p).
\]
From this, one may establish a weaker variant of Theorem~\ref{thm:super_fried} by appealing to Bourgain's result~\cite{Bourgain} (see also~\cite{kindler2018gaussian}).
The version of that theorem from~\cite[Theorem 3.5]{kindler2018gaussian} asserts
that if the Fourier weight of a function $f$ above level $k$ is at most $\xi/\sqrt{k}$, then the function is $(1+o(1))\xi$-close
to a $J = 2^{O(k)}/\xi^4$-junta. To use this result, one may observe that the noise stability bound above implies that for $k = 1/\eps$,
the Fourier weight of $f$ above level $k$ is at most $O_p(A \eps^{p}) = O_p\left(\frac{A \eps^{p-1/2}}{\sqrt{k}}\right)$, hence
by~\cite[Theorem 3.5]{kindler2018gaussian} $f$ is $O(A\eps^{p-1/2})$-close to a $J = \frac{2^{O(1/\eps)}}{A\eps^{p-1/2}}$-junta for all $\eps>0$.
In particular, setting $\eps = \Theta((\delta/A)^{2/(2p-1)})$, we get that $f$ is $\delta$-close to a $J = 2^{O((A/\delta)^{2/(2p-1)})}$ junta.
\end{remark}

\subsection{Proof Overview}
In this section we give an outline of the proof of Theorems~\ref{thm:tal_iso} and~\ref{thm:EG}. Both results rely on the same simple observation, that, once combined
with off-the-shelf Fourier concentration results, finishes the proof.

Given $f\colon \power{n}\to\{0,1\}$, we consider its Fourier coefficients
$\widehat{f}(S) = \Expect{x\in\power{n}}{f(x)\prod\limits_{i\in S}(-1)^{x_i}}$ as well as its gradient
$\nabla f\colon \power{n}\to\{-1,0,1\}^n$ defined as $\nabla f(x) = (f(x\oplus e_i) - f(x))_{i=1,\ldots,n}$.
We first observe that the square root of the sensitivity of $f$ at $x$ is nothing but $\norm{\nabla f(x)}_2$,
and also that looking at the absolute value of the $i$th entry in $\nabla f$, namely at the function
$x\rightarrow \card{f(x\oplus e_i) - f(x)}$, one has that its average is at least the absolute value of the
singleton Fourier coefficient $\card{\widehat{f}(\{i\})}$. This suggests that if $f$ has significant singleton
Fourier coefficients then the magnitude of the gradient should be large, and therefore the expected root of the sensitivity
should also be large. Indeed, by a simple application of the triangle inequality one gets that
\[
\Expect{x\in\power{n}}{\norm{\nabla f(x)}_2}
\geq
\norm{\Expect{x\in\power{n}}{\card{\nabla f(x)}}}_2
\geq
\left\|\left(\card{\widehat{f}(\{1\})},\ldots,\card{\widehat{f}(\{n\})}\right)\right\|_2,
\]
which is equal to $\sqrt{\sum\limits_{i=1}^{n}\widehat{f}(\{i\})^2}$. As the sum of all the squares of all of the Fourier coefficients
is at most $1$, we get that
$\Expect{x\in\power{n}}{\norm{\nabla f(x)}_2}\geq \sum\limits_{i=1}^{n}\widehat{f}(\{i\})^2$. In words, via a simple application of the
triangle inequality we managed to show a lower bound of the quantity we are interested in by the level $1$ Fourier weight of the function
$f$. This is already a rather interesting isoperimetric consequence, albeit rather weak.

We next show how to bootstrap this basic result into Theorems~\ref{thm:tal_iso} and~\ref{thm:EG} via random restrictions.
To get some intuition, suppose that we knew our function $f$ has significant weight on level $d$ (where $d > 1$); can we
utilize the above logic to say anything of interest about lower bounds of $\Expect{x\in\power{n}}{\norm{\nabla f(x)}_2}$? The answer is yes,
and one way to do that is via random restrictions. If $f$ has considerable weight on level $d$, we choose a subset of coordinates
$J\subseteq [n]$ by independently including each coordinate in it with probability $p = 1/d$ and then restrict the coordinates inside $\overline{J}$ to a
uniformly chosen value $z\in \power{\overline{J}}$. In that case, we expect the weight of $f$ to collapse from level $d$ to roughly level $pd$ , which by the choice of the parameter $p$ is roughly the first Fourier level. At that point
we could apply the above logic on the restricted function and get some lower bound for the expected root of the sensitivity of the restricted
function. But how does this last quantity relate to $\Expect{x\in\power{n}}{\norm{\nabla f(x)}_2}$?

Thinking about it for a moment, for a fixed point $x$, if we choose $J$ randomly as above and fix all coordinates inside $\overline{J}$, then only roughly
$1/d$ of the coordinates sensitive on $x$ remain alive, hence we expect the sensitivity to drop by a factor of $1/d$ as a result of this restriction.
Indeed, it can be shown that under random restrictions, the quantity $\Expect{x\in\power{n}}{\norm{\nabla f(x)}_2}$ drops by a factor of $1/\sqrt{d}$.
Thus, combining this observation with the above reasoning suggests that we should get a lower bound of
\[
\Expect{x\in\power{n}}{\norm{\nabla f(x)}_2}\geq \Omega\left(\sqrt{d}\sum\limits_{d\leq \card{S} \leq 2d}\widehat{f}(S)^2\right),
\]
and indeed
this is the bound we establish.

Theorems~\ref{thm:tal_iso} and~\ref{thm:EG} then follow from known Fourier tail bounds. Specifically, it is known that
(1) if $f$ has small variance then most of its Fourier mass lies above level $d=\Omega(\log(1/{\sf var}(f)))$,
and (2) if $f$ has small influences then most Fourier mass lies above level $d = \Omega(\log(1/\sum\limits_{i=1}^{n}I_i[f]^2))$.


\section{Preliminaries}
\paragraph{Notation.} We will use big $O$ notations throughout the paper. Sometimes, it will be easier for us to use
$\ll$ and $\gg$ notations; when we say that $X\ll Y$ we mean that there is an absolute constant $C>0$ such that $X\leq C Y$.
Analogously, when we say that $X\gg Y$ we mean that there is an absolute constant $C>0$ such that $X\geq C Y$.

\subsection{Fourier Analysis over the Hypercube}
We consider the space of real-valued functions $f:\power{n} \to {\mathbb R}$,
equipped with the inner product $\inner{f}{g} = \Expect{x\in_R\power{n}}{f(x)g(x)}$.
It is well-known that the collection of functions $\set{\chi_S}_{S\subseteq[n]}$ defined by
$\chi_S(x) = \prod\limits_{i\in S}{(-1)^{x_i}}$ forms an orthonormal basis, so any function $f$
may be written as $f(x) = \sum\limits_{S\subseteq[n]}{\widehat{f}(S)\chi_S(x)}$ in a unique way,
where
$\widehat{f}(S) = \inner{f}{\chi_S}$.
We also consider $L_p$ norms for $p \geq 1$,
that are defined as $\norm{f}_p= \left(\Expect{x}{\card{f(x)}^p}\right)^{1/p}$.

For $z\in\power{n}$, we denote by $z_{-i}\in\power{n-1}$ the vector resulting from dropping
the $i$th coordinate of $z$. We denote by $(x_i = 1, x_{-i} = z_{-i})$ the $n$-bit vector
which is $1$ on the $i$th coordinate and is equal to $z$ on any other coordinate.
\begin{definition}
  The derivative of $f$ along direction $i\in[n]$
  is $\partial_i f\colon\power{n}\to\mathbb{R}$ defined by
  $\partial_i f(z) = f(x_i = 0, x_{-i} = z_{-i}) - f(x_i = 1, x_{-i} = z_{-i})$.
  The gradient of $f$, $\nabla f\colon\power{n}\to\mathbb{R}^n$, is defined
  as $\nabla f(z) = (\partial_1 f(z),\ldots,\partial_n f(z))$.
\end{definition}
Note that for a Boolean function $f$, for each $x\in\power{n}$, $\partial_i f(x) \neq 0$ if and only if the coordinate $i$
is sensitive on $x$, and in this case its value is either $1$ or $-1$. Thus, $\norm{\nabla f(x)}_2 = \sqrt{s_f(x)}$,
and it will be often convenient for us to use the $\ell_2$-norm of the gradient of $f$ in place of $\sqrt{s_f(x)}$.

\begin{fact}\label{fact:singeton}
  Let $f\colon\power{n}\to\mathbb{R}$, and let $i\in[n]$.
  Then $\partial_i f(x) =
  2\sum\limits_{S\ni i}{\widehat{f}(S)\chi_{S\setminus \set{i}}(x)}$.
  In particular, we have that $\Expect{x}{\partial_i f(x)} = 2\widehat{f}(\set{i})$.
\end{fact}

Next, we define the level $d$ \emph{Fourier weight} of a function $f$, the \emph{approximate level $d$ weight} of a function $f$
and the \emph{Fourier weight of $f$ up to/ above level $d$}.
\begin{definition}
  The level $d$ Fourier weight of $f$ is defined to be $W_{=d}[f] = \sum\limits_{\card{S} = d}{\widehat{f}(S)^2}$. The
  approximate level $d$ weight of a function $f$ is defined to be $W_{\approx d}[f] = \sum\limits_{d\leq j<2d}{W_{=j}[f]}$.
  The weight of $f$ up to level $d$ is defined as $W_{\leq d}[f] = \sum\limits_{\card{S}\leq d}{\widehat{f}(S)^2}$,
  and the weight of $f$ above level $d$ is defined as $W_{> d}[f] = \sum\limits_{\card{S}> d}{\widehat{f}(S)^2}$.
\end{definition}
For a function $f\colon\power{n}\to\mathbb{R}$, we also define the \emph{level at most $d$ part} of $f$, denoted by $f^{\leq d}$,
as $f^{\leq d}(x) = \sum\limits_{\card{S}\leq d} \widehat{f}(S)\chi_S(x)$. We note that by Parseval's equality we have that
$\norm{f^{\leq d}}_2^2 = W_{\leq d}[f]$.
\subsection{Random Restrictions}
For a function $f\colon\{0,1\}^n\to\mathbb{R}$, a set of coordinates $J\subseteq[n]$ and $z\in\power{\overline{J}}$, we define the
\emph{restriction} of
$f$ according to $(J,z)$ as the function $f_{\overline{J}\rightarrow z}\colon \{0,1\}^{J}\to\mathbb{R}$ defined by
\[
f_{\overline{J}\rightarrow z}(y) = f(x_{\overline{J}} = z, x_{J} = y).
\]
We refer to the coordinates of $J$ as ``alive'', and refer to the coordinates of $\overline{J}$ as fixed.
We will often be interested in random restrictions of a function $f$. By that, we mean that the set $J$ is chosen randomly by including in it each
$i\in[n]$ independently with some probability (which is specified), and $z\in\power{\overline{J}}$ is chosen uniformly.
We have the following fact, which establishes a relation between
the approximate level $d$ weight of $f$ and the level $1$ weight of a random restriction of $f$ in which each $i\in [n]$ is included in $J$ with probability
$\frac{1}{2d}$ (see~\cite[Fact 4.1]{kindler2018gaussian} and~\cite[Section 3.1]{KKKMS}).
\begin{fact}\label{fact:restriction}
  Let $f\colon\power{n}\to\mathbb{R}$ and $d\in\mathbb{N}$, let $(J,z)$ be a random restriction
  where each $j\in [n]$ is included in $J$ with probability $\frac{1}{2d}$, and $z\in\power{\bar{J}}$ is chosen uniformly. Then $\Expect{J,z}{W_{=1}[f_{\bar{J}\rightarrow z}]}\gg W_{\approx d}[f]$.
\end{fact}
\begin{proof}
  For a fixed $J$ and $T\subseteq J$, we have that
  $\widehat{f_{\bar{J}\rightarrow z}}(T) = \sum\limits_{S\subseteq \bar{J}}\widehat{f}(T\cup S)\chi_S(z)$. Thus,
  \[
  \Expect{J,z}{W_{=1}[f_{\bar{J}\rightarrow z}]}
  =\Expect{J}{\Expect{z}{\sum\limits_{j\in J}\left(\sum\limits_{S\subseteq\bar{J}}\widehat{f}(S\cup \set{j})\chi_S(z)\right)^2}}
  =\Expect{J}{\sum\limits_{j\in J}\sum\limits_{S\subseteq\bar{J}}\widehat{f}(S\cup \set{j})^2},
  \]
  where the last transition follows by pushing the expectation over $z$ inside and using Parseval's equality. It follows that
  \[
  \Expect{J,z}{W_{=1}[f_{\bar{J}\rightarrow z}]} =
  \Expect{J}{\sum\limits_{T}\widehat{f}(T)^21_{\card{T\cap J} = 1}}
  =\sum\limits_{T}\widehat{f}(T)^2\Prob{J}{\card{T\cap J}=1},
  \]
  and as $\Prob{J}{\card{T\cap J}=1}\gg 1$ for all $T$ whose size is between $d$ and $2d$ and the rest of the summands
  are all non-negative, the proof is concluded.
\end{proof}

\subsection{Bonami's Lemma}
The \emph{degree} of a function $f\colon \power{n}\to\mathbb{R}$ is defined as
${\sf deg}(f) = \max\sett{\card{S}}{\widehat{f}(S)\neq 0}$.
We will need the following consequence of
the Bonami-Beckner-Gross hypercontractive inequality~\cite{Bec75,Bonami,grossLogSob} (see also~\cite{ODonnell}), showing that the $4$-norm of a low-degree function $f$ is comparable to its $2$-norm.
\begin{thm}\label{thm:hypercontractivity}
  Let $f\colon\power{n}\to\mathbb{R}$ be a function of degree at most $d$. Then
  $\norm{f}_4\leq \sqrt{3}^d\norm{f}_2$.
\end{thm}

\section{Isoperimetric Inequalities on the Hypercube}\label{sec:classic_iso}
In this section, we prove Theorems~\ref{thm:tal_iso} and~\ref{thm:EG}. We begin with the following very crude bound on the Talagrand boundary,
and then improve it using random restrictions.
\begin{lemma}\label{lemma:tal_lvl1}
  Let $f\colon\power{n}\to\mathbb{R}$. Then
  $\Expect{}{\norm{\nabla f(x)}_2 }\geq 2\sqrt{W_{=1}[f]}$.
\end{lemma}
\begin{proof}
By the triangle inequality and Fact~\ref{fact:singeton} we have
  \[
  \Expect{x}{\norm{\nabla f(x)}_2 }
  \geq \norm{\Expect{x}{\nabla f(x)}}_2
  = \norm{2(\widehat{f}(\set{1}),\ldots,\widehat{f}(\set{n}))}_2
  = 2\sqrt{\sum\limits_{i=1}^{n}\widehat{f}(\set{i})^2}
  = 2\sqrt{W_{=1}[f]}.
  \qedhere
 \]
\end{proof}

\begin{lemma}\label{lemma:tal_lvld_exact}
  Let $d\in\mathbb{N}$, and $f\colon\power{n}\to\power{}$. Then
  $\Expect{}{\norm{\nabla f(x)}_2 }\gg \sqrt{d}W_{\approx d}[f]$.
\end{lemma}
\begin{proof}
  Let $(J,z)$ be a random restriction for $f$, where each $j\in[n]$ is included in $J$ independently with probability
  $\frac{1}{2d}$, and $z\in\power{\bar{J}}$ is sampled uniformly. Fix $x\in \power{n}$, and note that
  \[
    \Expect{J}{\norm{\nabla f_{\bar{J}\rightarrow x_{\bar{J}}}(x_J)}_2}
    \leq \sqrt{\Expect{J}{\norm{\nabla f_{\bar{J}\rightarrow x_{\bar{J}}}(x_J)}^2_2}}
    =\sqrt{\Expect{J}{\sum\limits_{j}{\card{\partial_j f(x)}^2 1_{j\in J}}}}
    =\frac{\norm{\nabla f(x)}_2}{\sqrt{2d}}.
  \]
  Therefore,
  \begin{equation}\label{eq1}
  \frac{1}{\sqrt{2d}}\Expect{}{\norm{\nabla f(x)}_2}\geq \Expect{J,z}{\Expect{x\in\power{J}}{\norm{\nabla f_{\bar{J}\rightarrow z}(x)}_2}}.
  \end{equation}
  On the other hand, for each $J$ and $z\in\power{\bar{J}}$ we may apply Lemma~\ref{lemma:tal_lvl1} on $f_{\bar{J}\rightarrow z}$ to get that
  \[
  \Expect{x\in\power{J}}{\norm{\nabla f_{\bar{J}\rightarrow z}(x)}_2}\gg \sqrt{W_{=1}[f_{\bar{J}\rightarrow z}]}\geq W_{=1}[f_{\bar{J}\rightarrow z}].
  \]
  Taking expectation over $J$ and $z$ we get that
  \[
  \Expect{J,z}{\Expect{x\in\power{J}}{\norm{\nabla f_{\bar{J}\rightarrow z}(x)}_2}}
  \gg \Expect{J,z}{W_{=1}[f_{\bar{J}\rightarrow z}]}
  \gg W_{\approx d}[f],
  \]
  where the last inequality is by Fact~\ref{fact:restriction}. Combining this with equation~\eqref{eq1} finishes the proof.
\end{proof}

\begin{lemma}\label{lemma:tal_lvld_above}
  Let $d\in\mathbb{N}$, and $f\colon\power{n}\to\power{}$. Then
  $\Expect{}{\norm{\nabla f(x)}_2 }\gg \sqrt{d}W_{\geq d}[f]$.
\end{lemma}
\begin{proof}
  Applying Lemma~\ref{lemma:tal_lvld_exact} with $d 2^{j}$ in place of $d$ for $j=0,1,2,\ldots$, we get that
  \[
    2^{-j/2}\Expect{}{\norm{\nabla f(x)}_2 }
    \gg \sqrt{d}W_{\approx 2^{j} d}[f],
  \]
  and summing over $j$ yields the result.
\end{proof}

\subsection{Proof of Theorem~\ref{thm:tal_iso}}
  If ${\sf var}(f)\geq 2^{-16}$, then it is enough to prove a lower bound of $\gg {\sf var}(f)$.
  Applying Lemma~\ref{lemma:tal_lvld_exact} on
  $d$'s that are powers of $2$ between $1$ and $n$, we get that
  \[
  \sum\limits_{j=0}^{\log n}\frac{1}{\sqrt{2^{j}}}{\Expect{}{\norm{\nabla f(x)}_2 }}
  \gg\sum\limits_{j=0}^{\log n}{W_{\approx 2^j}[f]}
  ={\sf var}(f),
  \]
  and by geometric sum, the left hand side is $\ll \Expect{}{\norm{\nabla f(x)}_2 }$.

  We thus assume that ${\sf var}(f)\leq 2^{-16}$, and without loss of generality the majority value of $f(x)$ is $0$.
  Thus, $\E[f]\leq 2{\sf var}(f)$. Setting $d = \frac{1}{8} \log(1/{\sf var}(f))$, we have that
  \[
  \sum\limits_{0<\card{S}\leq d}\widehat{f}(S)^2
  \leq
  \norm{f^{\leq d}}_2^2
  =\inner{f}{f^{\leq d}}
  \leq \norm{f}_{4/3}\norm{f^{\leq d}}_4
  \leq \sqrt{3}^d \E[f]^{3/4}\norm{f}_2
  \leq 2^d\E[f]^{5/4}
  \leq
  0.9 {\sf var}(f),
  \]
  where we used the bound $\norm{f^{\leq d}}_4\leq\sqrt{3}^d\norm{f^{\leq d}}_2\leq \sqrt{3}^d\norm{f}_2$ which
  follows from Theorem~\ref{thm:hypercontractivity} and Parseval's equality.
  Therefore, $W_{\geq d}[f]\geq 0.1 {\sf var}(f)$, and by Lemma~\ref{lemma:tal_lvld_above} we conclude the result.\qed

\subsection{Proof of Theorem~\ref{thm:EG}}
Denote $M[f] = \sum\limits_{i=1}^{n} I_i[f]^2$; we need the following result
due to~\cite{Tal3,BKS,KK} (the precise version we use is due to Keller and Kindler~\cite{KK}).
\begin{thm}\label{thm:KK}
There are $c_1,c_2>0$, such that for all $f\colon \power{n}\to\power{}$ we have that
\[
W_{\leq c_1\log(1/M[f])}[f]\leq M[f]^{c_2}.
\]
\end{thm}

We may now give the proof of Theorem~\ref{thm:EG}.
\begin{proof}[Proof of Theorem~\ref{thm:EG}]
  If $M[f]\geq {\sf var}(f)^{2/c_2}$, then the result follows from
  Theorem~\ref{thm:tal_iso}. Otherwise, fix $c_1,c_2$ from Theorem~\ref{thm:KK} and set $d = c_1\log(1/M[f])$.
  By Theorem~\ref{thm:KK}
  we have that
  $W_{\leq d}[f]\leq M[f]^{c_2}\leq {\sf var}(f)^2\leq \frac{1}{2} {\sf var}(f)$,
  and so $W_{>d}[f]\geq \frac{1}{2} {\sf var}(f)$, and the result follows from
  Lemma~\ref{lemma:tal_lvld_above}.
\end{proof}

We remark that Theorem~\ref{thm:EG} implies a stability result for the KKL theorem. Namely,
if all influences of a Boolean function are at most $O(\log n/n)$, then $f$ must have constant-size
vertex boundary. More precisely:
\begin{corollary}\label{cor:stab}
  For any $K>0$ there are $c,\delta>0$ such that the following holds.
  Suppose that for $f\colon\power{n}\to\power{}$ we have that $\max_i I_i[f]\leq K{\sf var}(f)\frac{\log n}{n}$.
  Then
  \[
  \Prob{x}{\norm{\nabla f(x)}_2\geq c{\sf var}(f)\sqrt{\log n}}\geq \delta{\sf var}(f).
  \]
\end{corollary}
\begin{proof}
  Note that by assumption, $M[f]\leq K^2 \frac{\log^2 n}{n}\leq \frac{K^2}{\sqrt{n}}$,
  so Theorem~\ref{thm:EG} implies that
  $Z\defeq \Expect{x}{\norm{\nabla f(x)}_2}\geq K' {\sf var}(f) \sqrt{\log n}$ for some $K'$ depending only on
  $K$. Also, we have that
  \[
  \Expect{x}{\norm{\nabla f(x)}_2^2} = I_1[f]+\ldots+I_n[f]\leq K {\sf var}(f) \log n.
  \]
  It follows from the Paley-Zygmund inequality that
  \[
  \Prob{}{\norm{\nabla f(x)}_2\geq \frac{1}{2} Z}
  \geq \frac{1}{4}\frac{Z^2}{\Expect{x}{\norm{\nabla f(x)}_2^2}}
  \geq \frac{1}{4}\frac{K'^2}{K} {\sf var}(f),
  \]
  and the claim is proved for $c = \frac{1}{2} K'$ and $\delta = \frac{1}{4}\frac{K'^2}{K}$.
\end{proof}

\subsection{Extensions}
\subsubsection{Other $L_p$-norms}
Lemma~\ref{lemma:tal_lvl1} applies not only in the case of $L_2$-norms, but rather for any $L_p$ norm:
\begin{lemma}\label{lemma:tal_lvl1_L_p}
  Let $1\leq p\leq 2$ and $f\colon\power{n}\to\power{}$. Then
  $\Expect{}{\norm{\nabla f(x)}_p }\geq W_{=1}[f]^{1/p}$.
\end{lemma}
\begin{proof}
By Jensen's inequality and Lemma~\ref{lemma:tal_lvl1} we have
  \[
  \Expect{x}{\norm{\nabla f(x)}_p }
  =\Expect{x}{\norm{\nabla f(x)}_2^{2/p}}
  \geq \Expect{x}{\norm{\nabla f(x)}_2}^{2/p}
  \geq W_{=1}[f]^{1/p}.
  \qedhere
  \]
\end{proof}
As in Lemma~\ref{lemma:tal_lvld_exact}, the previous lemma quickly leads to the following conclusion.
\begin{lemma}\label{lemma:tal_lvld_above_alpha}
  Let $d\in\mathbb{N}$, $\alpha\in[1/2, 1]$ and $f\colon\power{n}\to\power{}$.
  Then $\Expect{}{s_f(x)^{\alpha}}\gg d^{\alpha} W_{\geq d}[f]$.
\end{lemma}

Using the same argument as before and replacing invocations of Lemma~\ref{lemma:tal_lvld_above} with Lemma~\ref{lemma:tal_lvld_above_alpha}, one may conclude variants of Theorems~\ref{thm:tal_iso} and~\ref{thm:EG} in which square roots are replaced with any power $p\in [1/2,1]$.
Below are precise statements.
\begin{thm}
  There exists $c>0$ such that for all non-constant functions $f\colon\power{n}\to\power{}$ and for all $p\in [1/2, 1]$ we have that
  \[
  \Expect{x}{s_f(x)^p}\geq c\cdot {\sf var}(f)\left(\log\left(\frac{1}{{\sf var}(f)}\right)\right)^{p}.
  \]
\end{thm}
\begin{proof}
  The proof is identical to the proof of Theorem~\ref{thm:tal_iso} where we use Lemma~\ref{lemma:tal_lvld_above_alpha}
  instead of Lemma~\ref{lemma:tal_lvld_above}.
\end{proof}

\begin{thm}
  There exists $c>0$ such that for all non-constant functions $f\colon\power{n}\to\power{}$
  and for all $p\in [1/2, 1]$ we have that
  \[
  \Expect{x}{s_f(x)^p}\geq c\cdot {\sf var}(f)\left(\log\left(1+\frac{1}{\sum\limits_{i=1}^{n} I_i[f]^2}\right)\right)^p.
  \]
\end{thm}
\begin{proof}
  The proof is identical to the proof of Theorem~\ref{thm:EG} where we use Lemma~\ref{lemma:tal_lvld_above_alpha}
  instead of Lemma~\ref{lemma:tal_lvld_above}.
\end{proof}
\subsubsection{Talagrand Boundary and Noise Sensitivity}
We next demonstrate the connection between having small Talagrand-type boundary and being noise stable, a notion that we define next.

For a parameter $\rho\in [0,1]$ and a point $x\in \power{n}$, the distribution of $\rho$-correlated inputs with $x$, denoted as $y\sim_{\rho} x$,
is defined by the following randomized process. For each $i\in[n]$ independently, set $y_i = x_i$ with probability $\rho$, and otherwise sample $y_i\in\{0,1\}$ uniformly.
The operator $T_{\rho}\colon L_2(\power{n}) \to L_2(\power{n})$, known as the noise operator with noise rate
$1-\rho$, is defined as
\[
T_{\rho} f(x) = \Expect{y\sim_{\rho} x}{f(y)}.
\]
\begin{definition}
  For $\eps>0$, the noise stability of $f$ with noise rate $\eps$ is defined as ${\sf Stab}_{1-\eps}(f) = \inner{f}{T_{1-\eps} f}$.
\end{definition}

The relation established in Lemma~\ref{lemma:tal_lvld_above_alpha} between the Fourier weight of $f$ and $\Expect{}{s_f(x)^{\alpha}}$
for $1/2\leq\alpha\leq 1$ implies a relation between the latter quantity and noise stability, as follows:
\begin{corollary}\label{corr:noise_stable}
  There exists an absolute constant $C>0$, such that for any $\delta\in[0,1/2]$ and for any function $f\colon\power{n}\to\power{}$ satisfying
  $\Expect{}{s_f(x)^{\frac{1}{2}+\delta}}\leq A$, we have that ${\sf Stab}_{1-\eps}(f)\geq 1 - C\cdot A\cdot \eps^{\frac{1}{2} + \delta}\log(1/\eps)$.
\end{corollary}
\begin{proof}
  \[
  1-{\sf Stab}_{1-\eps}(f) = \Prob{x,y\text{ $(1-\eps)$-correlated}}{f(x)\neq f(y)}
  =2\inner{1-f}{T_{1-\eps} f}
  =2\sum\limits_{k=0}^n{(1-(1-\eps)^{k})W_{=k}[f]}.
  \]

  Set $d = 1/\eps$; we split the sum into $k< d$ and $k \geq  d$, and upper bound the contribution of each part separately.
  For $k\geq d$, using Lemma~\ref{lemma:tal_lvld_above_alpha} we have
  \[
  2\sum\limits_{k=d}^n{(1-(1-\eps)^{k})W_{=k}[f]}
  \leq 2W_{\geq d}[f]\ll A\cdot d^{-(1/2+\delta)} = A\eps^{1/2+\delta}.
  \]
  For $k\leq d$, we write
  \[
  \sum\limits_{k=0}^d{(1-(1-\eps)^{k})W_{=k}[f]}
  \leq\sum\limits_{k=1}^d{k\eps W_{=k}[f]}
  \leq \eps \sum\limits_{k=1}^{d} W_{\geq k}[f].
  \]
  Using Lemma~\ref{lemma:tal_lvld_above_alpha} we get $W_{\geq k}[f]\ll A\cdot k^{-(1/2+\delta)}$, and so
  \[
  \sum\limits_{k=0}^d{(1-(1-\eps)^{k})W_{=k}[f]}
  \ll A\cdot \eps \sum\limits_{k=1}^{d} \frac{1}{k^{1/2+\delta}}
  \ll A \cdot \eps d^{1/2-\delta} \sum\limits_{k=1}^{d} \frac{1}{k}
  \ll A \cdot \eps d^{1/2-\delta} \log d.
  \qedhere
  \]
\end{proof}
\begin{remark}
  Inspecting the proof of Corollary~\ref{corr:noise_stable}, we see that if $\delta$ is bounded away from $1/2$, say
  $\delta\leq 0.49$, then the $\log(1/\eps)$ term may be removed. Indeed, this is true because in the end of the
  proof, we can replace the upper bound $\sum\limits_{k=1}^{d} \frac{1}{k^{1/2+\delta}}\ll d^{1/2-\delta}\log d$
  by the upper bound $\sum\limits_{k=1}^{d} \frac{1}{k^{1/2+\delta}}\ll d^{1/2-\delta}$ that holds if $\delta$ is
  bounded away from $1/2$.
\end{remark}

\subsection{Robust Versions}
Fix a $2$-coloring ${\sf col}\colon E(\{0,1\}^n)\to\{{\sf red}, {\sf blue}\}$.
Let $\nabla_{{\sf red}} f(x)$ be the Boolean vector whose $i$th coordinate is
$1$ if $f(x) = 1$ and $(x,x\oplus e_i)$ is a red sensitive edge, and let
$\nabla_{{\sf blue}} f(x)$ be the Boolean
vector whose $i$th coordinate is  $1$ if $f(x) = 0$ and $(x,x\oplus e_i)$ is a blue sensitive edge.
Analogously to Lemma~\ref{lemma:tal_lvl1} we have:
\begin{lemma}\label{lemma:tal_lvl1_robust}
  Let $f\colon\power{n}\to\power{}$, and ${\sf col}\colon E(\{0,1\}^n)\to\{{\sf red}, {\sf blue}\}$ be any
  coloring. Then
  \[
  \Expect{}{\norm{\nabla_{{\sf red}} f(x)}_2+\norm{\nabla_{{\sf blue}} f(x)}_2 }\geq \frac{1}{2}\sqrt{W_{=1}[f]}.
  \]
\end{lemma}
\begin{proof}
By the triangle inequality
$\Expect{x}{\norm{\nabla_{{\sf red}} f(x)}_2}
  \geq
  \norm{\Expect{x}{\nabla_{{\sf red}} f}(x)}_2$,
  and looking at the vector
  $v = \Expect{x}{\nabla_{{\sf red}} f(x)}$ we see its $i$th entry is
  $v_i = \Prob{x}{f(x) = 1, (x,x\oplus e_i)\text{ is a red influential edge}}$.
  Similarly, $\Expect{x}{\norm{\nabla_{{\sf blue}} f(x)}_2}
  \geq
  \norm{\Expect{x}{\nabla_{{\sf blue}} f}(x)}_2$ and looking at the vector
  $u = \Expect{x}{\nabla_{{\sf blue}} f(x)}$ we see its $i$th entry is
  $u_i = \Prob{x}{f(x) = 0, (x,x\oplus e_i)\text{ is a blue influential edge}}$.
  It follows that
  \[
  \Expect{}{\norm{\nabla_{{\sf red}} f(x)}_2+\norm{\nabla_{{\sf blue}} f(x)}_2 }
  \geq \norm{v}_2+\norm{u}_2
  \geq \norm{u+v}_2.
  \]
  Let $R_{i}(x)$ be the indicator that $(x,x\oplus e_i)$ is a red influential edge and $f(x) = 1$,
  and $B_i(x)$ be the indicator that $(x,x\oplus e_i)$ is a blue influential edge and $f(x) = 0$.
  Then
  \[
  u_i + v_i
  =\frac{1}{2^{n+1}}\sum\limits_{x}(R_i(x) + R_i(x\oplus e_i) + B_i(x) + B_i(x\oplus e_i))
  \geq
  \frac{1}{2^{n+1}}\sum\limits_{x} 1_{(x,x\oplus e_i)\text{ influential}}
  = \frac{1}{2}I_i[f].
  \]
  Hence
  \[
  \norm{u+v}_2
  \geq \frac{1}{2}\sqrt{\sum\limits_{i=1}^{n} I_i[f]^2}
  \geq \frac{1}{2} \sqrt{W_{=1}[f]}.
  \qedhere
  \]
  Secondly, for any $f$, coloring ${\sf col}$ and $x\in\power{n}$, choosing $J\subseteq [n]$ randomly by including each coordinate independently
  with probability $\frac{1}{2d}$,
  one has
  \[
  \Expect{J}{\sqrt{s_{f_{\bar{J}\rightarrow x_{\bar{J}}},{\sf red}}(x_{J})}}
  \leq
  \sqrt{\Expect{J}{s_{f_{\bar{J}\rightarrow x_{\bar{J}}},{\sf red}}(x_{J})}}
  =\sqrt{\frac{s_{f,{\sf red}}(x)}{2d}},
  \]
  and similarly for blue sensitivity. This implies the analog of~\eqref{eq1} for red and blue sensitivity, and repeating
  the argument in Lemmas~\ref{lemma:tal_lvld_exact},~\ref{lemma:tal_lvld_above} we get that
  ${\sf Talagrand}_{{\sf col}}(f)\gg \sqrt{d} W_{\geq d}[f]$ for any coloring ${\sf col}$. The robust versions of
  Theorems~\ref{thm:tal_iso} and~\ref{thm:EG} now follow as we have shown that
  $W_{\geq d}[f]\geq \Omega({\sf var}(f))$ where we can either take $d=\Omega(\log(1/{\sf var}(f)))$
  or $d=\Omega(\log(1/M(f)))$.
\end{proof}

\section{Improved Junta Theorems}\label{sec:super_fried}
\subsection{Proof of Theorem~\ref{thm:super_fried}}
In this section, we prove Theorem~\ref{thm:super_fried}. As explained in the introduction, if one does not care about getting tight dependency
between the junta size $J$ and the parameter $A$, then one may use Corollary~\ref{corr:noise_stable} and Bourgain's noise sensitivity theorem~\cite{Bourgain,kindler2018gaussian}
and get a junta theorem with slightly worse parameters than those stated in Theorem~\ref{thm:super_fried}. The main goal of this section
is to get an optimal dependency of $J$ on $A$.

For the proof, we will need the notion of noisy influences of a function $f$, and we first generalize the notion of influences to real valued functions.
\begin{definition}
  Let $f\colon\power{n}\to\mathbb{R}$ be a function, and let $i\in [n]$ be a coordinate.
  We define the influence of $i$ on $f$ as $I_i[f] = \norm{\partial_i f}_2^2$, where
  we recall that $\partial_i f$ is the discrete derivative of $f$ along direction $i$.
\end{definition}
\begin{definition}
  The $\rho$-noisy influence of a coordinate $i$ for $f\colon\{0,1\}^n\to\mathbb{R}$ is defined to be
  $I_i^{(\rho)}[f] = I_i[T_{\rho} f]$.
\end{definition}

\begin{fact}\label{fact:sum_of_noisy}[\cite{ODonnell}]
  For all $f\colon\{0,1\}^n\to\{0,1\}$ we have that $\sum\limits_{i=1}^{n} I_i^{(\rho)}[f] \ll \frac{1}{1-\rho}$.
  \qed
\end{fact}

Let $C_1, C_2$ be two absolute constants, and we think of $C_2$ as much larger than $C_1$. Given $f\colon\{0,1\}^n\to\{0,1\}$ with
$\Expect{x}{s_f(x)^p}\leq A$, we denote $d = C_1\left(\frac{A}{\eps}\right)^{1/p}$, $\delta = 2^{-\frac{C_2}{2p-1} d}$.
Define the candidate set of junta coordinates as
\[
\mathcal{J} = \sett{i\in [n]}{I_i[T_{1-\frac{1}{3d}}[f]]\geq \delta},
\]
i.e.~$\mathcal{J}$ is the set of coordinates whose noisy influence at $f$ is somewhat large. Our task is to show that $f$ is $\eps$-close to a $\mathcal{J}$ junta,
i.e.~that
\begin{equation}\label{eq6}
\sum\limits_{S: S\not\subseteq \mathcal{J}}\widehat{f}(S)^2\leq \eps,
\end{equation}
and the rest of the proof is devoted towards this goal. We remark that the junta $\mathcal{J}$ is
similar in spirit to the junta in Bourgain’s original proof~\cite{Bourgain}, which is easier to see in~\cite[Section 4]{kindler2018gaussian} and
in particular in~\cite[Theorem 4.20]{kindler2018gaussian}. Indeed, both here and in~\cite[Section 4]{kindler2018gaussian}, the junta is
taken to be the set of coordinates with significant noisy influence. There are many other similarities between the argument
presented therein and the argument here. First, in both cases one considers a dyadic sort of partitioning of the Fourier coefficients:
in~\cite[Section 4]{kindler2018gaussian} it is achieved via considering the difference in the noise sensitivity of $f$ at different
noise rates (see~\cite[Theorem 4.16]{kindler2018gaussian}). In the current proof, we perform explicit dyadic partitioning and also
accompany it by applications of suitable noise operators (the operators $S_k$, $S_k'$ and $S_k''$ below). Second, both proofs also use some relation between the noise sensitivity / moments
of the sensitivity and the level $1$ weight of the function $f$ (and its restrictions). While the technical execution of the proofs is
different, it is hard to pinpoint the high level difference between them.
For example, the ``harsh'' degree truncations
in~\cite{kindler2018gaussian} have to be replaced by the ``softer'' type of degree truncations that are achieved by the operators $S_k$, $S_k'$ and
$S_k''$ (see also the parameter $\tilde{W}_{k,d}$) to make the argument go through. Additionally, towards the end of the argument we are able to bound the sum of noisy influences in a non-trivial way, which ultimately allows us to choose $\delta = 2^{O_p(d)}$ as opposed to
$\delta = 2^{O_{p}(d\log d)}$ (which would have resulted in slightly worse bounds as in~\cite[Theorem 4.20]{kindler2018gaussian}); see Claim~\ref{claim:bound_low_deg_inf}.

Towards proving~\eqref{eq6}, we partition the contribution of different scales of
$\card{S\cap \mathcal{J}}$ to the left hand side, and for each $k=1,\ldots,\log d$ define
\[
W_{k,d}[f] = \sum\limits_{\substack{S: \card{S}\leq d,\\ 2^{k-1}\leq \card{S\cap \bar{\mathcal{J}}}\leq 2^k}}\widehat{f}(S)^2.
\]
Thus, we may write
\begin{equation}\label{eq:super_frid1}
  \sum\limits_{S: S\not\subseteq \mathcal{J}}\widehat{f}(S)^2
  \leq
  W_{>d}[f]
  +
  \sum\limits_{k=1}^{\log d} W_{k,d}[f].
\end{equation}
The rest of the proof is devoted to bounding $W_{>d}[f]$ and $W_{k,d}[f]$ for each $k$.
From Lemma~\ref{lemma:tal_lvld_above_alpha} it follows that
\begin{equation}\label{eq:super_frid}
  W_{>d}[f]\leq O\left(\frac{A}{d^p}\right) \leq \frac{\eps}{16}
\end{equation}
for sufficiently large $C_1$. Next, we bound $W_{k,d}[f]$. Fix $k$, and consider the operator $S_k = T_{1-1/d}^{\mathcal{J}}T_{1-1/2^{k}}^{\bar{\mathcal{J}}}$,
i.e. the noise operator that applies $1/2^k$ noise on coordinates outside $\mathcal{J}$, and $1/d$ noise on coordinates of $\mathcal{J}$. Note that
\[
W_{k,d}[f] \ll W_{k,d}[S_k f],
\]
so it suffices to bound $W_{k,d}[S_k f]$. We partition $W_{k,d}[S_k f]$ according to contribution of each $j\in\bar{\mathcal{J}}$, getting:
\begin{align*}
W_{k,d}[S_k f]
=\sum\limits_{\substack{S: \card{S}\leq d,\\ 2^{k-1}\leq \card{S\cap \bar{\mathcal{J}}}\leq 2^k}} \widehat{S_k f}(S)^2
\leq \frac{1}{2^{k-1}} \sum\limits_{j\not\in\mathcal{J}}\sum\limits_{\substack{S\ni j: \card{S}\leq d,\\ 2^{k-1}\leq \card{S\cap \bar{\mathcal{J}}}\leq 2^k}} \widehat{S_k f}(S)^2
&\ll \frac{1}{2^{k-1}} \sum\limits_{j\not\in\mathcal{J}}\norm{S_k \partial_j f}_2^2\\
&\defeq \tilde{W}_{k,d}[f].
\end{align*}
Thus, we have that $W_{k,d}[f]\ll \tilde{W}_{k,d}[f]$, and the rest of the proof is devoted to upper bounding the right hand side.
Define the operators $S_{k}'$ and $S_k''$ as
\[
S_k' = T_{1-\frac{1}{3d}}^{\mathcal{J}}T_{1-\frac{1}{2^{k}}+\frac{2}{3d}}^{\bar{\mathcal{J}}},
\qquad
S_k'' = T_{1-\frac{2}{3d}}^{\mathcal{J}}T_{1-\frac{1}{2^{k}}+\frac{1}{3d}}^{\bar{\mathcal{J}}}.
\]

We establish the following claim using random restrictions and hypercontractivity:
\begin{claim}\label{claim:super_fried_main}
  For all $\tau\in (0,1)$, we have that
  \[
  2^k \tilde{W}_{k,d}[f]\leq
  C\cdot d^{1-p}\tau^{2p-1} \Expect{x}{s_f(x)^p} + C\cdot\tau^{-\frac{1}{100 d}}\sum\limits_{j\in\bar{\mathcal{J}}} \norm{S_k' \partial_j f}_2^{2+\frac{1}{100d}},
  \]
  where $C$ is an absolute constant.
\end{claim}
\begin{proof}
  Deferred to Section~\ref{sec:missing_proofs}.
\end{proof}

Given Claim~\ref{claim:super_fried_main}, one may pull out $\norm{S_k' \partial_j f}_2^{\frac{1}{100 d}}\leq \delta^{\frac{1}{200d}}$ outside
the sum, and bound the sum on the rest of the noisy influences by $O(d)$ to get a bound of the form
$2^k \tilde{W}_{k,d}[f]\ll d^{1-p}\tau^{2p-1} \Expect{x}{s_f(x)^p} + d\tau^{-\frac{1}{100 d}}\delta^{\frac{1}{200d}}$.
This idea is good enough to establish a junta theorem, yet it gives worse parameters.\footnote{One has to take
$\delta = 2^{-O(d\log d)}$, as opposed to $\delta = 2^{-O(d)}$ as we have taken. We remark that if one only cares
about getting a larger junta of size $2^{O(d\log d)}$, the proof greatly simplifies.} Instead, to get the tight result with respect
to the junta size, we will need a more careful bound on the sum of noisy influences outside $\mathcal{J}$. Specifically,
we use an improved bound of the sum of noisy influences that relates them to $W_{\ell,d}[f]$.
To state this bound we define
\[
\eps_{\ell,d} = \sum\limits_{\substack{\card{S}>d\\ 2^{\ell-1}\leq \card{S\cap \bar{\mathcal{J}}}\leq 2^{\ell}}} \widehat{f}(S)^2.
\]
 \begin{claim}\label{claim:bound_low_deg_inf}
    $\frac{1}{2^k}\sum\limits_{j\in\bar{\mathcal{J}}} \norm{S_k' \partial_j f}_2^{2}\ll \sum\limits_{\ell=1}^{\log n} 2^{-\card{\ell-k}}(W_{\ell,d}[f]+\eps_{\ell, d})$.
  \end{claim}
  \begin{proof}
    By definition,
    \[
    \frac{1}{2^k}\sum\limits_{j\in\bar{\mathcal{J}}} \norm{S_k' \partial_j f}_2^{2}
    =4\sum\limits_{S} \frac{\card{S\cap \bar{\mathcal{J}}}}{2^k} \left(1-\frac{1}{2^k}+\frac{2}{3d}\right)^{\card{S\cap \bar{\mathcal{J}}}-1}\left(1-\frac{1}{3d}\right)^{\card{S\cap \mathcal{J}}}\widehat{f}(S)^2.
    \]
    Partitioning the last sum according to $\card{S\cap \bar{\mathcal{J}}}$ and dyadically partitioning so that $2^{\ell-1}\leq \card{S\cap \bar{\mathcal{J}}} < 2^{\ell}$, we have that the last
    expression is at most
    \[
    \ll
    \sum\limits_{\ell=1}^{\log n} 2^{\ell-k} \left(1-\frac{1}{2^k}+\frac{2}{3d}\right)^{2^{\ell-1}}(W_{\ell,d}[f]+\eps_{\ell, d})
    \ll
    \sum\limits_{\ell=1}^{\log n} 2^{-\card{\ell-k}}(W_{\ell,d}[f]+\eps_{\ell, d}).\qedhere
    \]
  \end{proof}

    Combining Claims~\ref{claim:super_fried_main} and~\ref{claim:bound_low_deg_inf} we immediately get the following corollary:
\begin{corollary}\label{corr:bound_low_deg_inf}
  There exists an absolute constant $C>0$ such that for all $\tau>0$,
    \[
    \tilde{W}_{k,d}[f]\leq C\cdot 2^{-k} d^{1-p}\tau^{2p-1} \Expect{x}{s_f(x)^p}
    +C\cdot \tau^{-\frac{1}{100 d}}\delta^{\frac{1}{200 d}}
    \left(\sum\limits_{\ell=1}^{\log d} 2^{-\card{\ell-k}}W_{\ell,d}[f]+\sum\limits_{\ell=1}^{\log n} 2^{-\card{\ell-k}}\eps_{\ell, d}\right).
    \]
\end{corollary}

  We may now give an upper bound on $\sum\limits_{k=1}^{\log d} W_{k,d}[f]$.
  \begin{claim}\label{claim:bd_outside_junta}
  For sufficiently large absolute constant $C_1$ and sufficiently large $C_2$ in comparison to $C_1$, we have that
    $\sum\limits_{k=1}^{\log d} W_{k,d}[f]\leq \frac{\eps}{16}$.
  \end{claim}
  \begin{proof}
    Let $\tau>0$ to be determined. As $W_{k,d}[f]\ll \tilde{W}_{k,d}[f]$, by Corollary~\ref{corr:bound_low_deg_inf} we have
    \[
    \sum\limits_{k=1}^{\log d} W_{k,d}[f]
    \ll
    \sum\limits_{k=1}^{\log d}
    \left(2^{-k}d^{1-p}\tau^{2p-1} A
    +\tau^{-\frac{1}{100 d}}\delta^{\frac{1}{200 d}}
    \left(\sum\limits_{\ell=1}^{\log d} 2^{-\card{\ell-k}}W_{\ell,d}[f]+\sum\limits_{\ell=1}^{\log n} 2^{-\card{\ell-k}}\eps_{\ell, d}\right)
    \right).
    \]
    The first sum is clearly $\ll d^{1-p}\tau^{2p-1} A$. The second sum is at most
    \[
    \tau^{-\frac{1}{100 d}}\delta^{\frac{1}{200 d}}
    \sum\limits_{\ell=1}^{\log d}W_{\ell,d}[f] \sum\limits_{k} 2^{-\card{\ell-k}}
    \ll \tau^{-\frac{1}{100 d}}\delta^{\frac{1}{200 d}}\sum\limits_{\ell=1}^{\log d} W_{\ell,d}[f].
    \]
    Finally, the third sum is at most
    \[
    \tau^{-\frac{1}{100 d}}\delta^{\frac{1}{200 d}}
    \sum\limits_{\ell=1}^{\log n} \eps_{\ell, d}
    \sum\limits_{k=1}^{\log d} 2^{-\card{\ell-k}}
    \ll
    \tau^{-\frac{1}{100 d}}\delta^{\frac{1}{200 d}}
    \sum\limits_{\ell=1}^{\log n} \eps_{\ell, d}
    \ll
    \tau^{-\frac{1}{100 d}}\delta^{\frac{1}{200 d}}
    W_{>d}[f]
    \ll \tau^{-\frac{1}{100 d}}\delta^{\frac{1}{200 d}}\eps,
    \]
    where we used~\eqref{eq:super_frid}.
    We thus get that
    \[
    \sum\limits_{k=1}^{\log d} W_{k,d}[f]
    \leq
    C\cdot d^{1-p}\tau^{2p-1} A
    +C\cdot \tau^{-\frac{1}{100 d}}\delta^{\frac{1}{200 d}}\sum\limits_{k=1}^{\log d} W_{k,d}[f]
    +C\cdot \tau^{-\frac{1}{100 d}}\delta^{\frac{1}{200 d}}\eps
    \]
    for some absolute constant $C>0$. We choose $\tau = \delta^{1/4}$ and then $C_2$ large enough so that
    $C\tau^{-\frac{1}{100 d}}\delta^{\frac{1}{200 d}}\leq \frac{1}{2}$, to get that
    $\sum\limits_{k=1}^{\log d} W_{k,d}[f]\leq 2C\cdot d^{1-p}\delta^{(2p-1)/4} A + 2C\delta^{\frac{1}{400 d}}\eps$,
    which is at most $\frac{\eps}{16}$ for large enough $C_2$.
  \end{proof}

  We can now prove Theorem~\ref{thm:super_fried}.
  \begin{proof}[Proof of Theorem~\ref{thm:super_fried}]
    Combining~\eqref{eq:super_frid1},~\eqref{eq:super_frid} and Claim~\ref{claim:bd_outside_junta} we get that
    $\sum\limits_{S: S\not\subseteq \mathcal{J}}\widehat{f}(S)^2\leq \frac{\eps}{8}$.
    Define $h(x) = 1$ if $\sum\limits_{S: S\subseteq \mathcal{J}}\widehat{f}(S)\chi_S(x)\geq 1/2$ and $0$ otherwise.
    Clearly, $h$ is a $\mathcal{J}$-junta and we have that
    \[
    \norm{f-h}_1
    =
    \norm{f-h}_2^2
    \leq 4\left\|f-\sum\limits_{S: S\subseteq \mathcal{J}}\widehat{f}(S)\chi_S(x)\right\|_2^2
    =4\sum\limits_{S: S\not\subseteq \mathcal{J}}\widehat{f}(S)^2\leq \eps.
    \]
    Finally, note that as the sum of $\left(1-\frac{1}{3d}\right)$-noisy influences of $f$ is at most $O(d)$ by Fact~\ref{fact:sum_of_noisy}, we get that
    \[
    \card{\mathcal{J}}\leq O\left(\frac{d}{\delta}\right).
    \qedhere
    \]
  \end{proof}

\subsection{Proof of Claim~\ref{claim:super_fried_main}}\label{sec:missing_proofs}
For the proof of Claim~\ref{claim:super_fried_main} we need a version of the hypercontractive inequality with the noise operator, as follows:
\begin{lemma}\label{lemma:hypercontractivity_noise}[\cite{ODonnell}]
  Let $q>2$, and $\rho\leq \frac{1}{\sqrt{q-1}}$. Then for all $f\colon\{0,1\}^n\to\mathbb{R}$ we have
  $\norm{T_{\rho} f}_q\leq \norm{f}_2$.
\end{lemma}

We also need the following claim, asserting that noise only decreases Talagrand's boundary.
\begin{claim}\label{claim:noise_decreases_tal}
  For $f\colon \{0,1\}^n\to\mathbb{R}$ and any noise operator $S = T_{\rho_1}\otimes\ldots T_{\rho_n}$ we have
  $\Expect{x}{\norm{\nabla{S f}(x)}_p}\leq \Expect{x}{\norm{\nabla{f}(x)}_p}$.
\end{claim}
\begin{proof}
  \[\Expect{x}{\norm{\nabla (S f)(x)}_p}
  \leq
  \Expect{x}{\norm{S \nabla f(x)}_p}
  =\Expect{x}{\norm{\Expect{y\sim S x}{\nabla f(y)}}_p},
  \]
  and the result follows from the triangle inequality.
\end{proof}

  Take $(J,z)$ to be a random restriction so that each $j\in [n]$ is included in $J$ independently with probability $\frac{1}{3d}$,
  and consider $g = S_k'' f$.
  We calculate
  the contribution of $\bar{\mathcal{J}}$ to the level $1$ weight of a restriction of $g$:
  \begin{align*}
  d\Expect{J, z}{\sum\limits_{j\in J\cap \bar{\mathcal{J}}}\widehat{g_{\overline{J}\rightarrow z}}(\set{j})^2}
  =d\sum\limits_{j\in \overline{\mathcal{J}}}\Expect{J}{\sum\limits_{S}\widehat{g(S)}^2 1_{S\cap J = \{j\}}}
  &=d\sum\limits_{S}\frac{\card{S\cap \bar{\mathcal{J}}}}{3d}\left(1-\frac{1}{3d}\right)^{\card{S}-1}\widehat{g(S)}^2\\
  &=\frac{1}{3}\sum\limits_{S}\card{S\cap \bar{\mathcal{J}}}\left(1-\frac{1}{3d}\right)^{\card{S}-1}\widehat{g(S)}^2.
  \end{align*}
  Also,
  \begin{align*}
    \sum\limits_{j\not\in \mathcal{J}} \norm{S_k \partial_j f}_2^2
    =
    4\sum\limits_{S}
    \card{S\cap \bar{\mathcal{J}}}
    \left(1-\frac{1}{2^k}\right)^{-2}\rho_1^{2\card{S\cap \mathcal{J}}}
    \rho_2^{2\card{S\cap \bar{\mathcal{J}}}}\widehat{g(S)}^2,
  \end{align*}
  where $\rho_1 = \frac{1-1/d}{1-2/(3d)}$ and $\rho_2 = \frac{1-1/2^k}{1-1/2^k+1/(3d)}$. Using $\frac{a}{b}\leq \frac{a+c}{b+c}$ for $0\leq a\leq b$ and $c\geq 0$ we have
  \[
  \rho_1 \leq \frac{1-1/d+2/(3d)}{1-2/(3d)+2/(3d)} = 1-\frac{1}{3d},
  \qquad
  \rho_2 \leq \frac{1-1/2^k+(1/2^k-1/(3d))}{1-1/2^k+1/(3d)+(1/2^k-1/(3d))} = 1-\frac{1}{3d},
  \]
  and so
  \[
  \sum\limits_{j\not\in \mathcal{J}} \norm{S_k \partial_j f}_2^2\leq
  12\left(1-\frac{1}{d}\right)^{-2} d\Expect{J, z}{\sum\limits_{j\in J\cap \bar{\mathcal{J}}}\widehat{g_{\overline{J}\rightarrow z}}(\set{j})^2}
  \ll d\Expect{J, z}{\sum\limits_{j\in J\cap \bar{\mathcal{J}}}\widehat{g_{\overline{J}\rightarrow z}}(\set{j})^2}.
  \]
  Hence our
  goal of upper bounding the left hand side in Claim~\ref{claim:super_fried_main} reduces to
  upper bounding the contribution of $\mathcal{J}$ to the level $1$ of random restrictions of $g$. For $J, z$ define
 \[
  L_{J,z} = \sett{j\in J\cap \bar{\mathcal{J}}}{\card{\widehat{g_{\overline{J}\rightarrow z}}(\set{j})} \leq \tau},
  \qquad
  H_{J,z} = \sett{j\in J\cap \bar{\mathcal{J}}}{\card{\widehat{g_{\overline{J}\rightarrow z}}(\set{j})} > \tau}.
  \]
  Then
  \[
  \Expect{J, z}{\sum\limits_{j\in J\cap \bar{\mathcal{J}}}\widehat{g_{\overline{J}\rightarrow z}}(\set{j})^2}
  =\Expect{J, z}{\sum\limits_{j\in L_{J,z}}\widehat{g_{\overline{J}\rightarrow z}}(\set{j})^2}
  +\Expect{J, z}{\sum\limits_{j\in H_{J,z}}\widehat{g_{\overline{J}\rightarrow z}}(\set{j})^2}.
  \]

  \paragraph{Bounding the contribution from $L_{J,z}$.}
  For the first sum, we have
  \begin{align*}
  \sum\limits_{j\in L_{J,z}}\widehat{g_{\overline{J}\rightarrow z}}(\set{j})^2
  \leq
  \left(\sum\limits_{j\in L_{J,z}}\widehat{g_{\overline{J}\rightarrow z}}(\set{j})^2\right)^{p}
  &\leq \left(\sum\limits_{j\in L_{J,z}}\tau^{(2p-1)/p}\card{\widehat{g_{\overline{J}\rightarrow z}}(\set{j})}^{1/p}\right)^{p}\\
  &\leq \tau^{2p-1}\left(\sum\limits_{j}\card{\widehat{g_{\overline{J}\rightarrow z}}(\set{j})}^{1/p}\right)^{p}.
  \end{align*}
  In the last expression we have $\tau^{2p-1}$ times the $1/p$-norm of the vector $(\widehat{g_{\overline{J}\rightarrow z}}(\set{j}))_{j\in J}$.
  Thus, using the triangle inequality in a similar way to Lemma~\ref{lemma:tal_lvl1}, we may bound the last expression from above by
  \begin{equation}\label{eq3}
  \tau^{2p-1}\Expect{x}{\norm{\nabla{g_{\overline{J}\rightarrow z}}(x)}_{1/p}}.
  \end{equation}
  Taking expectation, we get that
  \[
  \Expect{J,z}{\sum\limits_{j\in L_{J,z}}\widehat{g_{\overline{J}\rightarrow z}}(\set{j})^2}
  \leq \tau^{2p-1}\Expect{J,z}{\Expect{x}{\norm{\nabla {g_{\overline{J}\rightarrow z}}(x)}_{1/p}}}
  = \tau^{2p-1}\Expect{x}{\Expect{J,z}{\norm{\nabla g_{\overline{J}\rightarrow z}(x)}_{1/p}}}.
  \]
  Fix $x,z$ and consider the vector $\nabla g(x,z)$.
  Note that
  \[
  \Expect{J}{\norm{\nabla g_{\overline{J}\rightarrow z}(x)}_{1/p}^{1/p}}
  =
  \Expect{J}{\sum\limits_{j}\card{\partial_j g(x,z)}^{1/p}1_{j\in J}}
  =\frac{1}{3d}\norm{\nabla g(x,z)}_{1/p}^{1/p},
  \]
  and it follows by Jensen's inequality that
  \[
  \Expect{J}{\norm{\nabla g_{\overline{J}\rightarrow z}(x)}_{1/p}}
  \leq
  \Expect{J}{\norm{\nabla g_{\overline{J}\rightarrow z}(x)}_{1/p}^{1/p}}^{p}
  \leq d^{-p} \norm{\nabla g(x,z)}_{1/p}.
  \]
  It follows that~\eqref{eq3} is upper bounded by $\tau^{2p-1}d^{-p}\Expect{x,z}{\norm{\nabla g(x,z)}_{1/p}}$.
  Using Claim~\ref{claim:noise_decreases_tal}, we have that this may be upper bounded by
  $\tau^{2p-1}d^{-p}\Expect{x,z}{\norm{\nabla f(x,z)}_{1/p}} = \tau^{2p-1}d^{-p} \Expect{x,z}{s_f(x,z)^p}$.

  \paragraph{Bounding the contribution from $H_{I,z}$.}
  We have
  \[
  \sum\limits_{j\in H_{J,z}}\widehat{g_{\overline{J}\rightarrow z}}(\set{j})^2
  \leq \tau^{-\frac{1}{100d}}\sum\limits_{j\in H_{J,z}}\card{\widehat{g_{\overline{J}\rightarrow z}}(\set{j})}^{2+\frac{1}{100d}}
  \leq \tau^{-\frac{1}{100d}}\sum\limits_{j\in J\cap\bar{\mathcal{J}}}\card{\widehat{g_{\overline{J}\rightarrow z}}(\set{j})}^{2+\frac{1}{100d}},
  \]
  and we next take expectation over $z$:
  \begin{equation}\label{eq5}
  \Expect{z}{\sum\limits_{j\in H_{J,z}}\widehat{g_{\overline{J}\rightarrow z}}(\set{j})^2}
  \leq
  \tau^{-\frac{1}{100d}}\sum\limits_{j\in J\cap\bar{\mathcal{J}}}\Expect{z}{\card{\widehat{g_{\overline{J}\rightarrow z}}(\set{j})}^{2+\frac{1}{100d}}}.
  \end{equation}
  Define the operator
  $T = T_{\rho_1}^{\mathcal{J}}T_{\rho_2}^{\overline{\mathcal{J}}}$
  where $\rho_1 = \frac{1-2/(3d)}{1-1/(3d)}$ and $\rho_2 = \frac{1-1/2^k+1/(3d)}{1-1/2^k+2/(3d)}$,
  and note that by the semi-group property of the noise operator (namely, the fact that $T_{\rho}T_{\rho'} = T_{\rho\rho'}$)
  we may write $g = T S_k' f$. Thus, $\widehat{g_{\overline{J}\rightarrow z}}(\set{j})$ is equal to
  \[
  \sum\limits_{S\subseteq \overline{J} \cup\{j\}, S\ni j} \widehat{g}(S)\chi_{S\setminus\{j\}}(z) =
  \sum\limits_{S\subseteq \overline{J} \cup\{j\}, S\ni j}
  \rho_1^{\card{S\cap \mathcal{J}}}
  \rho_2^{\card{S\cap \overline{\mathcal{J}}}}
  \widehat{S_k' f}(S)\chi_{S\setminus\{j\}}(z).
  \]
  Thus, using hypercontractivity, i.e.~Lemma~\ref{lemma:hypercontractivity_noise}, we have that
  \begin{align}\label{eq4}
  \Expect{z}{\card{\widehat{g_{\overline{J}\rightarrow z}}(\set{j})}^{2+\frac{1}{100d}}}
  &=\left\|\sum\limits_{S\subseteq \overline{J} \cup\{j\}, S\ni j}
  \rho_1^{\card{S\cap \mathcal{J}}}
  \rho_2^{\card{S\cap \overline{\mathcal{J}}}}
  \widehat{S_k' f}(S)\chi_{S\setminus\set{j}}(z)\right\|_{2+\frac{1}{100d}}^{2+\frac{1}{100d}}\notag\\
  &\leq \left\|\sum\limits_{S\subseteq \overline{J} \cup\{j\}, S\ni j}
  \rho_1^{\card{S\cap \mathcal{J}}}
  \rho_2^{\card{S\cap \overline{\mathcal{J}}}}
  \left(1-\frac{1}{3d}\right)^{-(\card{S}-1)}\widehat{S_k' f}(S)\chi_{S\setminus\set{j}}(z)\right\|_2^{2+\frac{1}{100d}}\notag\\
  &\leq \left\|\sum\limits_{S\ni j}\widehat{S_k' f}(S)\chi_{S\setminus\set{j}}(z)\right\|_2^{2+\frac{1}{100d}}.
  \end{align}
  In the last inequality, we used Parseval's equality and that fact that $\rho_1,\rho_2\leq 1-\frac{1}{3d}$.
  Note that
  \[
  \left\|\sum\limits_{S\ni j}\widehat{S_k' f}(S)\chi_{S\setminus\set{j}}(z)\right\|_2^{2+\frac{1}{100d}}
  \ll \norm{S_k' \partial_j f}_2^{2+\frac{1}{100d}},
  \]
  so plugging~\eqref{eq4} into~\eqref{eq5} yields that
  \[
  \Expect{z}{\sum\limits_{j\in H_{J,z}}\widehat{g_{\overline{J}\rightarrow z}}(\set{j})^2}\ll
  \tau^{-\frac{1}{100d}}\sum\limits_{j\in J\cap\bar{\mathcal{J}}}\norm{S_k' \partial_j f}_2^{2+\frac{1}{100d}}.
  \]
  Taking expectation over $J$ now gives that
  \[
  \Expect{z,J}{\sum\limits_{j\in H_{J,z}}\widehat{g_{\overline{J}\rightarrow z}}(\set{j})^2}
  \ll
  \tau^{-\frac{1}{100d}}\sum\limits_{j\in\bar{\mathcal{J}}} \norm{S_k' \partial_j f}_2^{2+\frac{1}{100d}} \Expect{J}{1_{j\in J}}
  = \frac{1}{3d}\tau^{-\frac{1}{100d}}\sum\limits_{j\in\bar{\mathcal{J}}} \norm{S_k' \partial_j f}_2^{2+\frac{1}{100d}}.
  \]
\qed


\section*{Acknowledgments} 
We thank an anonymous for many helpful comments and corrections to an earlier version of the paper which led to
considerable improvements of the paper in all aspects.

\bibliographystyle{amsplain}
\newcommand{\etalchar}[1]{$^{#1}$}


\begin{dajauthors}
\begin{authorinfo}[pgom]
  Ronen Eldan\\
  Professor\\
  Weizmann Institute of Science\\
  Rehovot, Israel\\
  roneneldan\imageat{}gmail\imagedot{}com \\
  \url{https://www.wisdom.weizmann.ac.il/~ronene/}
\end{authorinfo}
\begin{authorinfo}[johan]
  Guy Kindler\\
  Professor\\
  Hebrew University of Jerusalem\\
  Jerusalem, Israel\\
  wgkindler\imageat{}gmail\imagedot{}com \\
  \url{https://www.cs.huji.ac.il/~gkindler/}
\end{authorinfo}
\begin{authorinfo}[laci]
  Noam Lifshitz\\
  Assistant Professor\\
  Hebrew University of Jerusalem\\
  Jerusalem, Israel\\
  noamlifshitz\imageat{}gmail\imagedot{}com\\
  \url{https://mathematics.huji.ac.il/people/noam-lifshitz}
\end{authorinfo}
\begin{authorinfo}[andy]
  Dor Minzer\\
  Assistant Professor\\
  Massachusetts Institute of Technology\\
  Cambridge, MA, USA\\
  dminzer\imageat{}mit\imagedot{}edu\\
  \url{https://sites.google.com/view/dorminzer/home}
\end{authorinfo}
\end{dajauthors}

\end{document}